\documentclass[3p]{elsarticle}
\usepackage{amsmath,hyperref}
\usepackage{amssymb,amsfonts}
\usepackage{amsthm}
\usepackage{mathrsfs}
\usepackage{latexsym,bm}
\usepackage{setspace}
\newtheorem{thm}{Theorem}
\newtheorem{lem}[thm]{Lemma}
\newtheorem{prop}[thm]{Proposition}
\newtheorem{definition}{Definition}
\newdefinition{rmk}{Remark}
\newproof{pf}{Proof}
\newproof{pot}{Proof of Theorem \ref{thm}}

\journal{XXX}

\begin{document}

\begin{frontmatter}

\title{A comparison theorem under sublinear expectations and related limit theorems}

\author[mymainaddress]{Ning Zhang}
\ead{nzhang@mail.sdu.edu.cn}

\author[mysecondaryaddress]{Yuting Lan\corref{mycorrespondingauthor}}
\cortext[mycorrespondingauthor]{Corresponding author}
\ead{lan.yuting@mail.shufe.edu.cn}

\address[mymainaddress]{School of Mathematics, Shandong University, Jinan 250100, China}
\address[mysecondaryaddress]{School of Statistics and Management, Shanghai University of Finance and Economics,\\
Shanghai 200433, China}

\begin{abstract}
In this paper, on the sublinear expectation space, we establish a comparison theorem between independent and convolutionary random vectors, which states that the partial sums of those two sequences of random vectors are identically distributed. Under the sublinear framework, through the comparison theorem, several fundamental limit theorems for convolutionary random vectors are obtained, including the law of large numbers, the central limit theorem and the law of iterated logarithm.
\end{abstract}


\end{frontmatter}

\section{Introduction}
The theory of nonlinear probabilities and expectations is frequently and broadly applied in various fields, such as risk measurement, nonlinear stochastic calculus and statistical uncertainty modeling, etc. The arising of the theory is mainly motivated by the fact that numerous uncertain phenomenon cannot be modeled or interpreted reasonably under the linear assumption, which is well illustrated by Denis and Martini\cite{DLMC}, Gilboa\cite{GI}, Chen and Epstein\cite{CZEL}, Peng\cite{PS97,PS99,PS08a}, etc.

Above works of nonlinear theory indicate that limit theorems, such as laws of large numbers and central limit theorems, play crucial roles in the interpretation, which motivates the investigation of limit theorems under nonlinear probabilities and expectations. De Cooman and Miranda\cite{CGME}, Epstein and Schneider\cite{ELSM}, Maccheroni and Marinacci\cite{MFMM} and Chen et al.\cite{CZWP} derive laws of large numbers for non-additive probabilities. Peng\cite{PS08b,PS09} introduces the general sublinear expectation and initiates the concept of independent (Peng independent) and identically distributed random vectors on the sublinear expectation space. Within the framework, Peng\cite{PS10} achieves following two limit results:
\begin{thm}[Weak law of large numbers]\label{PLLN}
Suppose $\{X_i\}_{i=1}^{\infty}$ is a sequence of independent and identically distributed random variables on the sublinear expectation space $(\Omega,\mathcal{H},\mathbb{E})$. Then for the partial sum $S_n=\sum_{i=1}^{n}X_i$ and any continuous function $\varphi$ with linear growth condition, there is
\[\lim\limits_{n\rightarrow\infty}\mathbb{E}[\varphi(\frac{1}{n}S_n)]=\mathbb{E}[\varphi(\eta)],\]
where $\eta$ follows the maximal distribution on $(\Omega,\mathcal{H},\mathbb{E})$ characterized by $G(p)=\mathbb{E}[\langle p,X_1\rangle]$.
\end{thm}
\begin{thm}[Central limit theorem]\label{PCLT}
Under same assumptions of Theorem \ref{PLLN}, additionally assuming $\mathbb{E}[X_1]=\mathcal{E}[X_1]=0$, for any continuous function $\varphi$ with linear growth condition, there is
\[\lim\limits_{n\rightarrow\infty}\mathbb{E}[\varphi(\frac{1}{\sqrt n}S_n)]=\mathbb{E}[\varphi(X)],\]
where $X$ follows the $G$-normal distribution on $(\Omega,\mathcal{H},\mathbb{E})$ characterized by $G(A)=\frac{1}{2}\mathbb{E}[\langle AX_1,X_1\rangle]$.
\end{thm}
\begin{rmk}
To be hold, Theorem \ref{PLLN} requires $\mathbb{E}[{|X_1|}^2]<\infty$ and Theorem \ref{PCLT} requires $\mathbb{E}[{|X_1|}^{2+\delta}]<\infty$ for some $\delta>0$, which are indicated in the proofs of Peng\cite{PS10}.
\end{rmk}
Chen and Hu\cite{CZHF} also derive a law of iterated logarithm on sublinear expectation spaces:
\begin{thm}\label{CHLIN}
Under same assumptions of Theorem \ref{PCLT}, together with the assumption that $\mathbb{E}[X_1^2]=\overline{\sigma}^2$ and $\mathcal{E}[X_1^2]=\underline{\sigma}^2$, denoting $a_n=\sqrt{2n\log\log n}$, there is
\[v(\underline{\sigma}\leq\limsup\limits_{n\rightarrow\infty}\frac{S_n}{a_n}\leq\overline{\sigma})=1\quad
and \quad
v(-\overline{\sigma}\leq\liminf\limits_{n\rightarrow\infty}\frac{S_n}{a_n}\leq-\underline{\sigma})=1.\]
\end{thm}

In this paper, we focus on the notion of convolution, which is a weaker condition than the Peng independence. The main motivation of this paper is to establish similar limit theorems for convolutionary random vectors. Without certain properties, it's difficult to prove those limit theorems directly under the convolutionary condition. Inspired by the comparison theorem between negatively associated and independent
random variables on linear probability spaces obtained by Shao\cite{SQ}, we establish a comparison theorem between convolutionary and Peng independent random vectors under sublinear expectations. Then we extend the weak law of large numbers, central limit theorem and law of iterated logarithm from Peng independent to convolutionary cases through our comparison theorem without applying any additional assumption.

This paper is organized as following: Initially, we present some notions and properties of sublinear expectations and capacities in Section 2. Then we state and prove our main result, the comparison theorem, in Section 3. Finally, we establish the weak law of large numbers, central limit theorem and law of iterated logarithm in Section 4 as applications of the comparison theorem.
\section{Preliminary}
In this section, we introduce some fundamental definitions and lemmas about capacities and sublinear expectations. Let $(\Omega,\mathcal{F})$ be a measurable space throughout the paper.
\begin{definition}
A set function $V(\cdot)$ defined on the $\sigma$-field $\mathcal{F}$ is called a capacity if it satisfies:
\begin{enumerate}
\item $V(\emptyset)=0$ and $V(\Omega)=1$.
\item $V(A)\leq V(B)$ if $A \subseteq B$ and $A,B\in\mathcal{F}$.
\end{enumerate}
\end{definition}
\begin{definition}\label{vc}
A capacity $V(\cdot)$ on the $\sigma$-field $\mathcal{F}$ is continuous if it satisfies:
\begin{enumerate}
\item \textbf{Continuous from below:} $V(A_n)\uparrow V(A)$, whenever $A_n\uparrow A$ and $A,A_n\in\mathcal{F}$.
\item \textbf{Continuous from above:} $V(A_n)\downarrow V(A)$, whenever $A_n\downarrow A$ and $A,A_n\in\mathcal{F}$.
\end{enumerate}
\end{definition}
In the sequel, let $\mathcal{H}$ be a set of random variables on $(\Omega,\mathcal{F})$ such that if $X_1,X_2,\cdots,X_n\in \mathcal{H}$, then $\varphi(X_1,X_2,\cdots,X_n) \in \mathcal{H}$ for each $\varphi \in C_{l,Lip}(\mathbb{R}^n)$, where $C_{l,Lip}(\mathbb{R}^n)$ denotes the set of local Lipschitz continuous functions, which is the linear space of functions $\varphi$ satisfying
\[|\varphi(x)-\varphi(y)|\leq C(1+|x|^m+|y|^m)|x-y|,\quad \forall \, x,y \in \mathbb{R}^n,some\,\, C>0, m\in N.\]
In this case, $X=(X_1,X_2,\cdots,X_n)$ is called an $n$-dimensional random vector, denoted as $X \in \mathcal{H}^n$.
\begin{definition}[Peng\cite{PS09}]
A sublinear expectation $\mathbb{E}$ on $\mathcal{H}$ is a functional $\mathbb{E}:\mathcal{H}\rightarrow\mathbb{R}$ satisfying:
\begin{enumerate}
\item \textbf{Monotonicity:}\quad $\mathbb{E}[X]\geq \mathbb{E}[Y]$, $\forall \,X,Y\in\mathcal{H},\,X\geq Y$.
\item \textbf{Constant preserving:}\quad $\mathbb{E}[c]=c$, $\forall \, c\in\mathbb{R}$.
\item \textbf{Sub-additivity:}\quad $\mathbb{E}[X+Y]\leq \mathbb{E}[X]+\mathbb{E}[Y]$, $\forall \, X,Y\in\mathcal{H}$.
\item \textbf{Positive homogeneity:}\quad $\mathbb{E}[\lambda X]=\lambda\mathbb{E}[X]$, $\forall \lambda\geq 0,\,X \in \mathcal{H}$.
\end{enumerate}
$(\Omega,\mathcal{H},\mathbb{E})$ is called a sublinear expectation space. The conjugate of $\mathbb{E}$ is defined by:
$\mathcal{E}[X]=-\mathbb{E}[-X]$.
\end{definition}
Peng\cite{PS09} also initiates the definition of identically distributed and independent random vectors under sublinear expectations.
\begin{definition}
Let $X_1$ and $X_2$ be two $d$-dimensional random vectors defined on two sublinear expectation spaces $(\Omega_1,\mathcal{H}_1,\mathbb{E}_1)$ and $(\Omega_2,\mathcal{H}_2,\mathbb{E}_2)$ respectively. Then the two vectors $X_1$ and $X_2$ are called identically distributed, denoted by $X_1 \overset{d}{=}X_2$, if for each function $\varphi\in C_{l,Lip}(\mathbb{R}^d)$, there is
\[\mathbb{E}_1[\varphi(X_1)]=\mathbb{E}_2[\varphi(X_2)].\]

\end{definition}
\begin{definition}
Suppose $X$ and $Y$ are $d_1$-dimensional and $d_2$-dimensional random vectors defined on the sublinear expectation space $(\Omega,\mathcal{H},\mathbb{E})$. The random vector $Y$ is said to be independent from the random vector $X$, if for each function $\varphi\in C_{l,Lip}(\mathbb{R}^{d_1+d_2})$, there is
\[\mathbb{E}[\varphi(X,Y)]=\mathbb{E}[\mathbb{E}[\varphi(x,Y)]_{x=X}].\]
$\{X_i\}_{i=1}^{\infty}$ is called an independent sequence if $X_{i+1}$ is independent from $(X_1,\cdots,X_i)$ for any $i \in \mathbb{N}^+$.
\end{definition}
\begin{rmk}\label{rm1}
If $Y$ is independent from $X$ satisfying $X\geq0$ and $\mathbb{E}[Y]\geq0$, then there is $\mathbb{E}[XY]=\mathbb{E}[X]\cdot\mathbb{E}[Y]$.
\end{rmk}
\begin{definition}[Peng\cite{PS09}]
The product space of two sublinear expectation spaces $(\Omega_1,\mathcal{H}_1,\mathbb{E}_1)$ and $(\Omega_2,\mathcal{H}_2,\mathbb{E}_2)$, denoted as
$(\Omega_1\times\Omega_2,\mathcal{H}_1\otimes\mathcal{H}_2,\mathbb{E}_1\otimes\mathbb{E}_2)$, is defined by
\[\mathcal{H}_1\otimes\mathcal{H}_2:=\{Z(\omega_1,\omega_2)=\varphi(X(\omega_1),Y(\omega_2)):\,\,(\omega_1,\omega_2)\in\Omega_1\times\Omega_2,(X,Y)\in\mathcal{H}_1^m\times\mathcal{H}_2^n,\,\,\varphi\in C_{l,Lip}(\mathbb{R}^{m+n})\}\]
and for random variables of the form $Z(\omega_1,\omega_2)=\varphi(X(\omega_1),Y(\omega_2))$, there is
\[\mathbb{E}_1\otimes\mathbb{E}_2[Z]:=\mathbb{E}_1[\mathbb{E}_2[\varphi(x,Y)]_{x=X}].\]
By the same method, we can define the product space $(\prod\limits_{i=1}^n\Omega_i,\bigotimes\limits_{i=1}^n\mathcal{H}_i,\bigotimes\limits_{i=1}^n\mathbb{E}_i)$ of given sublinear expectation spaces $(\Omega_i,\mathcal{H}_i,\mathbb{E}_i)$, $i=1,2,\cdots,n$. Furthermore, if the random variable $Z$ belongs to the product space $(\prod\limits_{i=1}^k\Omega_i,\bigotimes\limits_{i=1}^k\mathcal{H}_i,\bigotimes\limits_{i=1}^k\mathbb{E}_i)$ for some positive integer $k<\infty$, then $Z$ belongs to the product space $(\prod\limits_{i=1}^\infty\Omega_i,\bigotimes\limits_{i=1}^\infty\mathcal{H}_i,\bigotimes\limits_{i=1}^\infty\mathbb{E}_i)$ with
$\bigotimes\limits_{i=1}^\infty\mathbb{E}_i[Z]:=\bigotimes\limits_{i=1}^k\mathbb{E}_i[Z]$. It's easy to verify that the product spaces defined above are also sublinear expectation spaces.
\end{definition}
\begin{prop}\label{pr1}
Let $X_i$ be any $d_i$-dimensional random vector on sublinear expectation space $(\Omega_i,\mathcal{H}_i,\mathbb{E}_i)$ respectively, $i=1,2,\cdots,n$. Denote
\[Y_i(\omega_1,\omega_2,\cdots,\omega_n)=X_i(\omega_i),\quad i=1,2,\cdots,n.\]
Then $Y_i$ is a $d_i$-dimensional random vector on the sublinear expectation space $(\prod\limits_{i=1}^n\Omega_i,\bigotimes\limits_{i=1}^n\mathcal{H}_i,\bigotimes\limits_{i=1}^n\mathbb{E}_i)$. Moreover, $Y_i \overset {d}{=} X_i$ and $Y_{i+1}$ is independent from $(Y_1,\cdots,Y_i)$ for each $i=1,2,\cdots,n-1$.
\end{prop}
Peng\cite{PS10} further introduces the maximal and G-normal distribution on the sublinear expectation space.
\begin{definition}\label{rm:1}
A $d$-dimensional random vector $\eta$ on a sublinear expectation space $(\Omega,\mathcal{H},\mathbb{E})$ is maximal distributed if there exists a bounded, closed and convex subset $\Gamma\subset\mathbb{R}^d$ such that
\[\mathbb{E}[\varphi(\eta)]=\max\limits_{x\in\Gamma}\varphi(x),\quad \forall \varphi\in C_{l,Lip}(\mathbb{R}^{d}).\]
If $\eta$ is $1$-dimensional with $\mathbb{E}[\eta]=\overline{\mu}$ and $\mathcal{E}[\eta]=\underline{\mu}$, then we have $\Gamma=[\overline{\mu},\underline{\mu}]$.
\end{definition}

\begin{definition}
A $d$-dimensional random vector $X$ on a sublinear expectation space $(\Omega,\mathcal{H},\mathbb{E})$ is G-normal distributed if it satisfies
\[aX+b\hat{X}\overset{d}{=}\sqrt{a^2+b^2}X,\quad\forall\,a,b\geq 0,\]
where $\hat X\overset{d}{=}X$ and $\hat X$ is independent from $X$.
\end{definition}
The maximal and G-normal distribution can also be characterized by following:
\begin{lem}\label{lm:3}
Let $(X_i,\eta_i),\,i=1,2$ be two pairs of $d_1$-dimensional G-normal distributed and $d_2$-dimensional maximal distributed random vectors on sublinear expectation spaces $(\Omega_i,\mathcal{H}_i,\mathbb{E}_i),\,i=1,2$, respectively. Then
\begin{enumerate}
\item $X_1\overset{d}{=}X_2$ if and only if for any $d_1\times d_1$ symmetric matrix $A$, there is $\mathbb{E}_1[\langle AX_1,X_1\rangle]=\mathbb{E}_2[\langle AX_2,X_2\rangle]$.

\item $\eta_1\overset{d}{=}\eta_2$ if and only if for any $p\in\mathbb{R}^{d_2}$, there is $\mathbb{E}_1[\langle p,\eta_1\rangle]=\mathbb{E}_2[\langle p,\eta_2\rangle]$.

\end{enumerate}
\end{lem}

Next, we introduce the concept of convolution, which allows random vectors to be weakly dependent in the sense of Peng independence.
\begin{definition}
Suppose $X$ and $Y$ are two d-dimensional random vectors on the sublinear expectation space $(\Omega,\mathcal{H},\mathbb{E})$. Then $Y$ and $X$ are said to be of convolution, if for each function $\varphi\in C_{l,Lip}(\mathbb{R}^{d})$, there is
\[\mathbb{E}[\varphi(X+Y)]=\mathbb{E}[\mathbb{E}[\varphi(x+Y)]_{x=X}].\]
A sequence of random vectors $\{X_i\}_{i=1}^{\infty}$ is called a convolutionary sequence, if $\sum\limits_{i=k}^nX_i$ and $X_{n+1}$ are of convolution for any $n \in \mathbb{N}^+$ and $k\leq n$.
\end{definition}
From the definition, it's straightforward that if $\{X_i\}_{i=1}^{\infty}$ is an independent sequence, then it's a convolutionary one. However, if $Y$ and $X$ are of convolution, the property mentioned in Remark \ref{rm1} does not necessarily hold since the function $\varphi(x,y)=xy$ can not be represented by any function of $x+y$. In the following section, we establish an equivalence between the partial sums of the two sequences.
\section{The comparison theorem}
Inspired by the comparison theorem between negatively dependent and independent random variables under classic linear probability proposed by Shao\cite{SQ}, in this section, we provide a comparison theorem between convolutionary and Peng independent random vectors such that some limit results for Peng independent random vectors can be generalized to the convolutionary situation.
\begin{thm}\label{tm:1}
Let $\{X_i\}_{i=1}^{\infty}$ be a convolutionary sequence of $d$-dimensional  random vectors on the sublinear expectation space $(\Omega,\mathcal{H},\mathbb{E})$. Then there exists an independent sequence of $d$-dimensional random vectors $\{\hat{X}_i\}_{i=1}^{\infty}$ on a sublinear expectation space $(\hat{\Omega},\hat{\mathcal{H}},\hat{\mathbb{E}})$ such that  $\hat{X}_i\overset{d}{=}X_i$ for $i\in \mathbb{N}^{+}$ and
\begin{equation}\label{eqc}
\mathbb{E}[\varphi(\sum\limits_{i=m+1}^{m+n} X_i)]= \hat{\mathbb{E}}[\varphi(\sum\limits_{i=m+1}^{m+n} \hat{X}_i)],\quad \forall m \in \mathbb{N}, n \in \mathbb{N}^{+}, \varphi \in C_{b,Lip}(\mathbb{R}^d),
\end{equation}
where $C_{b,Lip}(\mathbb{R}^d)$ denotes bounded and Lipschitz continuous functions on $\mathbb{R}^d$, which is a subset of $C_{l,Lip}(\mathbb{R}^d)$.
\end{thm}
\begin{proof}
To begin with, we verify the existence of $\{\hat{X}_i\}_{i=1}^{\infty}$. According to Proposition \ref{pr1}, for the convolutionary sequence $\{X_i\}_{i=1}^{\infty}$ on $(\Omega,\mathcal{H},\mathbb{E})$, define $\{\hat{X}_i\}_{i=1}^{\infty}$ on the sublinear expectation space $(\Omega^{\infty},\mathcal{H}^{\otimes \infty},\mathbb{E}^{\otimes \infty})$ by
\[\hat{X}_i(\omega_1,\omega_2,\cdots)=X_i(\omega_i),\quad i=1,2,\cdots.\]
Then $\{\hat{X}_i\}_{i=1}^{\infty}$ is a sequence of Peng independent random vectors with $\hat{X}_i\overset{d}{=}X_i$ for $i\in \mathbb{N}^{+}$.\\
Next we will prove Eq. \eqref{eqc}. Without loss of generality, we only consider the case of $m=0$. When $n=1$, the equation holds trivially. Considering $n=2$, for any $\varphi \in C_{l,Lip}(\mathbb{R}^d)$, we have
\[\mathbb{E}[\varphi(X_1+X_2)]=\mathbb{E}[\mathbb{E}[\varphi(x+X_2)]_{x=X_1}].\]
Then denote
\[\varphi_i(x)=\mathbb{E}[\varphi(x+X_i)],\quad \psi_i(x)=\hat{\mathbb{E}}[\varphi(x+\hat{X}_i)],\quad i=2,3,\cdots.\]
It's straightforward that $\varphi_i(x)=\psi_i(x)$ since $X_i\overset{d}{=}\hat{X}_i$. For $\varphi \in C_{b,Lip}(\mathbb{R}^d)$, there exists a constant $C>0$ such that $|\varphi(x)-\varphi(y)|\leq C|x-y|$. It's obvious that for any $i$, functions $\varphi_i,\,\psi_i\in C_{b,Lip}(\mathbb{R}^d)$ since
\[
|\varphi_i(x)-\varphi_i(y)| =|\mathbb{E}[\varphi(x+X_i)]-\mathbb{E}[\varphi(y+X_i)]|
\leq \mathbb{E}[|\varphi(x+X_i)-\varphi(y+X_i)|]
\leq C|x-y|.\]
Consequently, we can obtain
\[
\mathbb{E}[\varphi(X_1+X_2)]=\mathbb{E}[\varphi_2(X_1)]=\mathbb{E}[\psi_2(X_1)]={\mathbb{E}}[\psi_2(\hat{X}_1)]
=\hat{\mathbb{E}}[\hat{\mathbb{E}}[\varphi(x+\hat{X}_2)]_{x=\hat{X}_1}]=\hat{\mathbb{E}}[\varphi(\hat{X}_1+\hat{X}_2)].\]
Furthermore, we assume that Eq. \eqref{eqc} holds for $n=n_0$, that is $\mathbb{E}[\varphi(\sum\limits_{i=1}^{n_0} X_i)]= \hat{\mathbb{E}}[\varphi(\sum\limits_{i=1}^{n_0} \hat{X}_i)]$.\\
Then consider the case of $n=n_0+1$. Denote the partial sum of $\{{X}_i\}_{i=1}^{\infty}$ and $\{\hat{X}_i\}_{i=1}^{\infty}$ by $S_n=\sum\limits_{i=1}^{n}X_i$ and $\hat S_n=\sum\limits_{i=1}^{n}\hat X_i$ respectively. We have
\begin{align*}
\mathbb{E}[\varphi(S_{n_0+1})] & = \mathbb{E}[\mathbb{E}[\varphi(x+X_{n_0+1})]_{x=S_{n_0}}]
=\mathbb{E}[\varphi_{n_0+1}(S_{n_0})]
= \hat{\mathbb{E}}[\varphi_{n_0+1}(\hat S_{n_0})]\\
&=\hat{\mathbb{E}}[\psi_{n_0+1}(\hat S_{n_0})]
=\hat{\mathbb{E}}[\hat{\mathbb{E}}[\varphi(x+\hat X_{n_0+1})]_{x=\hat S_{n_0}}] =\hat{\mathbb{E}}[\varphi(\hat S_{n_0+1})].
\end{align*}
Thus Eq. \eqref{eqc} also holds for $n=n_0+1$. Therefore by induction, Eq. \eqref{eqc} holds for any $n \in \mathbb{N}^{+}$ and $\varphi \in C_{b,Lip}(\mathbb{R}^d)$ with $m=0$. Eq. \eqref{eqc} can be proved for all $m \in \mathbb{N}$ through the same method.
\end{proof}
\section{Limit theorems for convolutionary random vectors}
In this section, we derive several limit theorems for convolutionary random vectors through the comparison theorem. Initially, we introduce the weak law of large numbers and central limit theorem for convolutionary random vectors. Then we demonstrate a law of iterated logarithm for convolutionary random variables.

Before any generalization, we introduce a lemma extending convergent results for bounded Lipschitz continuous functions to all continuous functions with linear growth condition.
\begin{lem}[Peng\cite{PS08b}]\label{lm:a1}
Let $(\Omega,{\mathcal{H}},{\mathbb{E}})$ and $(\hat{\Omega},\hat{\mathcal{H}},\hat{\mathbb{E}})$ be two sublinear expectation spaces. $Y_n\in\mathcal{H}$, $n=1,2,\cdots$ and $\hat Y\in\hat{\mathcal{H}}$ are $d$-dimensional random vectors satisfying that $\sup_{n}\mathbb{E}[|Y_n|^2]+\hat{\mathbb{E}}[|\hat Y|^2]<\infty$. If $\lim\limits_{n\rightarrow \infty}\mathbb{E}[\varphi(Y_n)]=\hat{\mathbb{E}}[\varphi(\hat Y)]$ holds for each $\varphi \in C_{b,Lip}(\mathbb{R}^d)$, then it also holds for all continuous functions $\varphi$ on $\mathbb{R}^d$ satisfying linear growth condition $|\varphi(x)|\leq C(1+|x|)$ for some $C>0$.
\end{lem}
\subsection{Weak law of large numbers}
\begin{thm}\label{nwlln}
Let $\{X_i\}_{i=1}^{\infty}$ be a convolutionary sequence of $d$-dimensional  identically distributed random vectors on sublinear expectation space $(\Omega,\mathcal{H},\mathbb{E})$ with $\mathbb{E}[|X_1|^2]<\infty$. Denoting the partial sum of $\{X_i\}_{i=1}^{\infty}$ by $S_n=\sum_{i=1}^{n}X_i$, we have
\begin{equation}\label{eq:5}
\lim\limits_{n\rightarrow\infty}\mathbb{E}[\varphi(\frac{1}{n}S_n)]=\mathbb{E}[\varphi(\eta)],
\end{equation}
where $\eta$ follows the maximal distribution characterized by $G(p)=\mathbb{E}[\langle p,X_1\rangle] $ and $\varphi$ is any continuous function on $\mathbb{R}^d$ satisfying linear growth condition.
\end{thm}
\begin{proof}
Firstly, we will prove that Eq. \eqref{eq:5} holds for all $\varphi \in C_{b,Lip}(\mathbb{R}^d)$. According to Theorem \ref{tm:1}, given the convolutionary sequence of random vectors $\{X_i\}_{i=1}^{\infty}$, we can find the corresponding Pend independent random vectors $\{\hat{X}_i\}_{i=1}^{\infty}$ on the product space $(\hat{\Omega},\hat{\mathcal{H}},\hat{\mathbb{E}}):=(\Omega^\infty,\mathcal{H}^{\otimes \infty},\mathbb{E}^{\otimes \infty})$. Then $\{\hat{X}_i\}_{i=1}^{\infty}$ are identically distributed with $\hat{\mathbb{E}}[\hat X_1^2]<\infty$. Applying Theorem \ref{PLLN}, there is
\[\lim\limits_{n\rightarrow\infty}\hat{\mathbb{E}}[\varphi(\frac{1}{n}\hat S_n)]=\hat{\mathbb{E}}[\varphi(\hat\eta)],\]
where $\hat S_n=\sum_{i=1}^{n}\hat X_i$ and $\hat\eta$ is a $d$-dimensional maximal distributed random vector on $(\hat{\Omega},\hat{\mathcal{H}},\hat{\mathbb{E}})$ characterized by $\hat G(p)=\hat{\mathbb{E}}[\langle p,\hat X_1\rangle] $. Suppose $\eta$ is a $d$-dimensional maximal distributed random vector on $(\Omega,\mathcal{H},\mathbb{E})$ characterized by $G(p)=\mathbb{E}[\langle p,X_1\rangle]=\hat{\mathbb{E}}[\langle p,\hat X_1\rangle]=\hat G(p) $, by Lemma \ref{lm:3}, we get
\[\mathbb{E}[\varphi(\eta)]=\hat{\mathbb{E}}[\varphi(\hat\eta)],\quad \forall \varphi\in C_{l,Lip}(\mathbb{R}^d).\]
For any $n \in \mathbb{N}^{+}$ and $\varphi\in C_{b,Lip}(\mathbb{R}^d)$, denotes $\varphi_n(x)=\varphi(\frac{1}{n}x)$. It is obvious that $\varphi_n(x)$ also belongs to $C_{b,Lip}(\mathbb{R}^d)$. Thus from Theorem \ref{tm:1}, we obtain
\[\lim\limits_{n\rightarrow\infty}{\mathbb{E}}[\varphi( \frac{1}{n}S_n)]=\lim\limits_{n\rightarrow\infty}\mathbb{E}[\varphi_n(S_n)]=\lim\limits_{n\rightarrow\infty}\hat{\mathbb{E}}[\varphi_n(\hat S_n)]={\mathbb{E}}[\varphi(\eta)].\]
Therefore we prove that Eq. \eqref{eq:5} holds for all $\varphi \in C_{b,Lip}(\mathbb{R}^d)$. \\
Since $\mathbb{E}[|X_1|^2]<\infty$ and $\mathbb{E}[|\eta|^2]<\infty$, there is
$\sup\limits_{n}\mathbb{E}[|\frac{1}{n}S_n|^2]+{\mathbb{E}}[|\eta|^2]<\infty$.
Then applying Lemma \ref{lm:a1}, we can extend the result for all continuous functions on $\mathbb{R}^d$ satisfying liner growth condition.
\end{proof}
\subsection{Central limit theorem}
By the similar method, we can establish the central limit theorem for convolutionary random vectors.
\begin{thm}
Let $\{X_i\}_{i=1}^{\infty}$ be a convolutionary sequence of $d$-dimensional identically distributed random vectors on sublinear expectation space $(\Omega,\mathcal{H},\mathbb{E})$ with $\mathbb{E}[X_1]=\mathcal{E}[X_1]=0$ and $\mathbb{E}[{|X_1|}^{2+\delta}]<\infty$ for some $\delta>0$. Then for the partial sum $S_n=\sum_{i=1}^{n}X_i$ and any linear growth function $\varphi$ in $C(\mathbb{R}^d)$, there is
\begin{equation}\label{eq:7}
\lim\limits_{n\rightarrow\infty}\mathbb{E}[\varphi(\frac{1}{\sqrt n}S_n)]=\mathbb{E}[\varphi(X)],
\end{equation}
where $X$ follows G-normal distribution characterized by
$G(A)=\frac{1}{2}\mathbb{E}[\langle AX_1,X_1\rangle]$.
\end{thm}
\begin{proof}
Similar to the proof of Theorem \ref{nwlln}, we can find the corresponding independent sequence of random vectors on $(\hat{\Omega},\hat{\mathcal{H}},\hat{\mathbb{E}}):=(\Omega^\infty,\mathcal{H}^{\otimes \infty},\mathbb{E}^{\otimes \infty})$ satisfying $\hat X_i \overset {d}{=} X_i,\,i=1,2,\cdots$. Then $\{\hat X_i\}_{i=1}^{\infty}$ satisfies the assumptions of Theorem \ref{PCLT}. Thus we have
\[\lim\limits_{n\rightarrow\infty}\hat{\mathbb{E}}[\varphi(\frac{1}{\sqrt n}\hat S_n)]=\hat{\mathbb{E}}[\varphi(\hat X)],\]
where $\hat S_n$ is the partial sum of $\{\hat X_i\}_{i=1}^{\infty}$ and $\hat X$ is a $d$-dimensional G-normal distributed random vector on $(\hat{\Omega},\hat{\mathcal{H}},\hat{\mathbb{E}})$ characterized by
$\hat G(A)=\frac{1}{2}\hat{\mathbb{E}}[\langle A\hat X_1,\hat X_1\rangle]$ for any $d\times d$ symmetric matrix $A$.\\
Suppose $X$ is a $d$-dimensional G-normal distributed random vector on $(\Omega,\mathcal{H},\mathbb{E})$ characterized by
\[G(A)=\hat G(A)=\frac{1}{2}\hat{\mathbb{E}}[\langle A\hat X_1,\hat X_1\rangle]=\frac{1}{2}\mathbb{E}[\langle AX_1,X_1\rangle].\]
Then, by Theorem \ref{tm:1} and Lemma \ref{lm:3}, we obtain that for any $\varphi \in C_{b,Lip}(\mathbb{R}^d)$,
\[\lim\limits_{n\rightarrow\infty}{\mathbb{E}}[\varphi( \frac{1}{\sqrt n}S_n)]=\lim\limits_{n\rightarrow\infty}\hat{\mathbb{E}}[\varphi(\frac{1}{\sqrt n}\hat S_n)]=\hat{\mathbb{E}}[\varphi(\hat X)]=\mathbb{E}[\varphi(X)],\]
Notice that $\mathbb{E}[{|X_1|}^{2+\delta}]<\infty$ and $\mathbb{E}[|X|^2]<\infty$, which implies that $\sup\limits_{n}\mathbb{E}[|\frac{1}{\sqrt n}S_n|^2]+{\mathbb{E}}[|\eta|^2]<\infty$. Finally applying Lemma \ref{lm:a1}, we can extend the result for all continuous functions $\mathbb{R}^d$ with liner growth condition.
\end{proof}
\subsection{Law of iterated logarithm}
Chen and Hu \cite{CZHF} establish a law of iterated logarithm for Peng independent and identically distributed random variables. In this part, we establish similar results for convolutionary random variables.\\
Let $\mathcal{Q}$ be a non-empty set of probability measures on $(\Omega,\mathcal{F})$. Same as Chen and Hu \cite{CZHF}, we also assume $\mathcal{H}\subset\bigcap_{P\in\mathcal{Q}}L^1(\Omega,\mathcal{F},P)$ and satisfies that if $X\in \mathcal{H}$, then the indicator function $I_{\{X\in A\}}$ belongs to $\mathcal{H}$ for any $ A \in \mathcal{B}(\mathbb{R})$. With these additional assumptions, we can define the sublinear expectation $\mathbb{E}$ and its conjugate expectation $\mathcal{E}$ on $\mathcal{H}$ by
\[\mathbb{E}[X]:=\sup\limits_{P\in\mathcal{Q}}E_P[X],\quad \mathcal{E}[X]:=\inf\limits_{P\in\mathcal{Q}}E_P[X],\quad \forall X\in \mathcal{H}\]
and the corresponding upper and lower capacity $\mathbb{V}$ and $v$ by
\[\mathbb{V}(A):= \sup\limits_{P\in\mathcal{Q}}P(A),\quad v(A):= \inf\limits_{P\in\mathcal{Q}}P(A),\quad \forall A\in\mathcal{F}.\]
According to Definition \ref{vc}, it is straightforward to check that the upper capacity $\mathbb{V}$ is continuous from below and the lower capacity $v$ is continuous from above. For the upper and lower capacities, we also have the Borel-Cantelli lemma.
\begin{lem}\label{lm:1}(Chen et al. \cite{CZWP}). Let $\{A_i\}^{\infty}_{i=1}$ be a sequence of events in $\mathcal{F}$. If $\sum_{i=1}^{\infty}\mathbb{V}(A_i)<\infty$, then there is $\mathbb{V}(\bigcap_{n=1}^{\infty}\bigcup_{i=n}^{\infty}A_i)=0$.
\end{lem}
Finally, we demonstrate the law of iterated logarithm for convolutionary random variables.
\begin{thm}\label{tm:3}
Suppose $\{X_i\}_{i=1}^{\infty}$ is a  convolutionary sequence of identically distributed random variables on $(\Omega,\mathcal{H},\mathbb{E})$ satisfying
$\mathbb{E}[X_1]=\mathcal{E}[X_1]=0 $,
 $\mathbb{E}[X_1^2]=\overline{\sigma}^2$ and $\mathcal{E}[X_1^2]=\underline{\sigma}^2$. With $a_n=\sqrt{2n\log\log n}$, there is
\begin{equation}\label{eq:8}
v\left(-\overline{\sigma}\leq\liminf\limits_{n\rightarrow\infty}\frac{S_n}{a_n}\leq\limsup\limits_{n\rightarrow\infty}\frac{S_n}{a_n}\leq\overline{\sigma}\right)=1.
\end{equation}
\end{thm}
To prove Theorem \ref{tm:3}, we first establish the following lemma for convolutionary random variables.
\begin{lem}\label{lm:4}
Under assumptions of Theorem \ref{tm:3}, denoting $S_{m,n}=\sum\limits_{i=m+1}^{m+n}X_i$, then for any $r>2$, there exists a positive constant $K_r$ not depending on $n$ such that
\[\mathbb{E}[\max\limits_{i\leq n}|S_{m,i}|^r]\leq K_rn^{\frac{r}{2}},\quad \forall\,\, m\in\mathbb{N}.\]
\end{lem}
\begin{proof}
Consider the corresponding independent sequence $\{\hat X_i\}_{i=1}^{\infty}$ on $(\hat{\Omega},\hat{\mathcal{H}},\hat{\mathbb{E}})$. Notice that Theorem \ref{tm:1} holds for all $\varphi \in C_{l,Lip}(\mathbb{R})$ under the additional assumptions of $\mathcal{H}$. Thus Applying Theorem \ref{tm:1}, we obtain
\[\sup\limits_{m \in\mathbb{N}}\mathbb{E}[|S_{m,n}|^r]=\sup\limits_{m \in\mathbb{N}}\hat{\mathbb{E}}[|\hat S_{m,n}|^r]\leq C_rn^{\frac{r}{2}},\]
where the last inequality follows from Theorem 3.1 in  Hu \cite{HF}.
Then according to the proof of Theorem 3.7.5 in Stout \cite{SW}, we obtain
\[\mathbb{E}[\max\limits_{i\leq n}|S_{m,i}|^r]\leq K_rn^{\frac{r}{2}},\quad \forall\,\, m\in\mathbb{N}.\]
\end{proof}
\begin{proof}[Proof of Theorem \ref{tm:3}]
Consider the corresponding independent sequence $\{\hat X_i\}_{i=1}^{\infty}$ on $(\hat{\Omega},\hat{\mathcal{H}},\hat{\mathbb{E}})$ to the convolutionary sequence $\{X_i\}_{i=1}^{\infty}$.
For independent $\{\hat X_i\}_{i=1}^{\infty}$, applying the Chebyshev's inequality and central limit theorem, Chen and Hu \cite{CZHF} derive that for any $\epsilon>0$, there is
\begin{equation}\label{eq:9}
\sum\limits_{n_k\geq m_0}\hat{\mathbb{V}}\left(\frac{\hat S_{n_k}}{a_{n_k}}>(1+\epsilon)\overline\sigma\right)<\infty,
\end{equation}
where $m_0$ is a sufficiently large integer and $n_k:=[e^{k^{\alpha}}]$ with some chosen $0<\alpha<1$ independent of $\epsilon$.\\
Then for any given $\epsilon>0$, we define a function $\phi\in C_{b,Lip}(\mathbb{R})$ by
\begin{equation}\label{eq:10}
\phi(x)=\begin{cases}
0            & x\leq(1+\frac{\epsilon}{2})\overline\sigma;       \\
\frac{2}{\epsilon\overline\sigma}(x-(1+\frac{\epsilon}{2})\overline\sigma)     & (1+\frac{\epsilon}{2})\overline\sigma<x\leq(1+\epsilon)\overline\sigma;       \\
1                & x>(1+\epsilon)\overline\sigma.
\end{cases}
\end{equation}
It's easy to verify that $I_{\{x>(1+\epsilon)\overline\sigma\}}<\phi(x)<I_{\{x>(1+\frac{\epsilon}{2})\overline\sigma\}}$.
Therefore we obtain that
\begin{align*}
\mathbb{V}\left(\frac{S_{n_k}}{a_{n_k}}>(1+\epsilon)\overline\sigma\right) & =\mathbb{E}\left[I\left\{\frac{S_{n_k}}{a_{n_k}}>(1+\epsilon)\overline\sigma\right\}\right]\leq\mathbb{E}\left[\phi\left(\frac{S_{n_k}}{a_{n_k}}\right)\right] =\hat{\mathbb{E}}\left[\phi\left(\frac{\hat S_{n_k}}{a_{n_k}}\right)\right]\\
&\leq\hat{\mathbb{E}}\left[I\left\{\frac{\hat S_{n_k}}{a_{n_k}}>\left(1+\frac{\epsilon}{2}\right)\overline\sigma\right\}\right]
=\hat{\mathbb{V}}\left(\frac{\hat S_{n_k}}{a_{n_k}}>\left(1+\frac{\epsilon}{2}\right)\overline\sigma\right)
\end{align*}
Since Eq. \eqref{eq:9} holds for arbitrary $\epsilon>0$, we have
\[\sum\limits_{n_k\geq m_0}\mathbb{V}\left(\frac{S_{n_k}}{a_{n_k}}>(1+\epsilon)\overline\sigma\right)\leq\sum\limits_{n_k\geq m_0}\hat{\mathbb{V}}\left(\frac{\hat S_{n_k}}{a_{n_k}}>\left(1+\frac{\epsilon}{2}\right)\overline\sigma\right)<\infty\]
Consequently, applying Lemma \ref{lm:1}, we can obtain $\mathbb{V}(\bigcap\limits_{m=1}^{\infty}\bigcup\limits_{k=m}^{\infty}\{\frac{S_{n_k}}{a_{n_k}}>(1+\epsilon)\overline\sigma\})=0$, which implies
\[v\left(\limsup\limits_{k\rightarrow\infty}\frac{S_{n_k}}{a_{n_k}}\leq(1+\epsilon)\overline\sigma\right)=1.\]
Then we define $M_k:=\max\limits_{n_k\leq n< n_{k+1}}\frac{|S_n-S_{n_k}|}{a_{n_k}}$ for $k\in \mathbb{N}^{+}$. For each $k\in \mathbb{N}^{+}$ and $n_k\leq n <n_{k+1}$, there is
\[\frac{S_{n}}{a_{n}}\leq \frac{S_{n_k}}{a_{n_k}}\frac{a_{n_k}}{a_n}+M_k\frac{a_{n_k}}{a_n}.\]
By choosing $p>2$ such that $p(1-\alpha)\geq 2$, according to Lemma \ref{lm:4}, we get
\[\sum\limits_{k=1}^{\infty}\mathbb{E}[M_k^p]\leq K_p\sum\limits_{k=1}^{\infty}\frac{(n_{k+1}-n_k)^{\frac{p}{2}}}{a_{n_k}^{p}}\leq C_p \sum\limits_{k=1}^{\infty}k^{-\frac{p(1-\alpha)}{2}}(\log k)^{-\frac{p}{2}}<\infty,\]
where $C_p$ is a positive constant. By Chebyshev's inequality, we achieve that for any $\epsilon>0$, there is
\[\sum\limits_{k=1}^{\infty}\mathbb{V}(M_k>\epsilon)\leq\sum\limits_{k=1}^{\infty}\frac{\mathbb{E}[M_k^p]}{\epsilon^p}<\infty.\]
Then, according to the Lemma \ref{lm:1}, we have
$v(\limsup\limits_{k\rightarrow\infty}M_k\leq\epsilon )=1$.\\
Thus by the arbitrariness of $\epsilon$ and the upper continuity of the lower capacity $v$, we get
\[v(\limsup\limits_{k\rightarrow\infty}M_k=0 )=1.\]
Combined with the fact that $\lim\limits_{k\rightarrow\infty}\frac{a_{n_k}}{a_{n_{k+1}}}=1$, we obtain that for arbitrary $\epsilon$
\[v\left(\limsup\limits_{k\rightarrow\infty}\frac{S_{n}}{a_{n}}\leq(1+\epsilon)\overline\sigma\right)\geq v\left(\left\{\limsup\limits_{k\rightarrow\infty}\frac{S_{n_k}}{a_{n_k}}\leq(1+\epsilon)\overline\sigma\right\}\cap\left\{\limsup\limits_{k\rightarrow\infty}M_k=0 \right\}\right)=1,\]
implying that
\begin{equation}\label{eq:11}
v\left(\limsup\limits_{k\rightarrow\infty}\frac{S_{n}}{a_{n}}\leq\overline\sigma\right)=1.
\end{equation}
Finally, by considering the sequence $\{-X_i\}_{i=1}^{\infty}$, we get
\begin{equation}\label{eq:12}
v\left(\limsup\limits_{k\rightarrow\infty}\frac{-S_{n}}{a_{n}}\leq\overline\sigma\right)=v\left(\liminf\limits_{k\rightarrow\infty}\frac{S_{n}}{a_{n}}\geq -\overline\sigma\right)=1.
\end{equation}
Combining Eq. \eqref{eq:11} with Eq. \eqref{eq:12}, we obtain the desired result
\[v\left(-\overline\sigma\leq\liminf\limits_{k\rightarrow\infty}\frac{S_{n}}{a_{n}}\leq\limsup\limits_{k\rightarrow\infty}\frac{S_{n}}{a_{n}}\leq\bar\sigma\right)=1.\]
\end{proof}

\end{document}